\documentclass[a4paper,12pt]{article}

\usepackage{amscd}

\usepackage{amsmath}
\usepackage{ntheorem}
\usepackage{amssymb}
\usepackage{latexsym}
\usepackage{url,mathrsfs}
\usepackage[all]{xy}\UseComputerModernTips
\usepackage{graphicx}

\theoremstyle{plain}
\newtheorem{proposition}{Proposition}[section]
\newtheorem{theorem}[proposition]{Theorem}
\newtheorem{lemma}[proposition]{Lemma}

\newtheorem{conjecture}[proposition]{Conjecture}

\theoremstyle{plain}
\theorembodyfont{\normalfont\rm}
\newtheorem{definition}[proposition]{Definition}
\newtheorem{remark}[proposition]{Remark}
\newtheorem{example}[proposition]{Example}

\theoremstyle{nonumberplain}
\theoremheaderfont{\normalfont\itshape}
\theoremseparator{.}
\theorembodyfont{\normalfont}
\newtheorem{proof}{Proof}

\newcommand{\qed}{\hfill $\Box$}

\makeatletter

\@addtoreset{equation}{section}
\makeatother

\newcommand{\mer}{{\rm mero}}
\newcommand{\merc}{{\rm mero,c}}

\newcommand{\ZZ}{{\mathbb Z}}

\newcommand{\RR}{{\mathbb R}}
\newcommand{\CC}{{\mathbb C}}

\newcommand{\e}{\varepsilon}

\newcommand{\Db}{{\bf D}^{\mathrm{b}}}
\newcommand{\Dbc}{{\bf D}_{\mathrm{c}}^{\mathrm{b}}}

\newcommand{\F}{{\cal F}}
\newcommand{\G}{{\cal G}}

\renewcommand{\L}{{\mathcal L}}

\renewcommand{\O}{{\cal O}}

\newcommand{\simto}{\overset{\sim}{\longrightarrow}}

\def\cf{cf.\kern.3em}
\def\eg{e.g.\kern.3em}

\title{On the Topology of Fibers of Rational Functions 
\\[0.5em]
\large To the memory of Mihai Tib{\u a}r
}

\author{Kiyoshi TAKEUCHI 
\footnote{Mathematical Institute, Tohoku University, 
Aramaki Aza-Aoba 6-3, Aobaku, Sendai, 980-8578, Japan. 
E-mail: takemicro@nifty.com} }
\date{\today}

\begin{document}

\maketitle

\begin{abstract}
In 1988 Broughton proved that if a complex polynomial 
function on $\CC^n$ is tame at infinity its fibers are 
homotopy equivalent to bouquets of some spheres. 
Since then several basic properties of polynomial functions 
were studied by many mathematicians in various 
situations. In this paper, we show some evidence for 
the possibility of extending Broughton's result 
to rational functions satisfying 
some tameness condition at infinity and the one at their 
points of indeterminacy. Specifically, assuming some 
conditions on a rational function on $\CC^n$ 
we prove that the cohomology groups of a 
generic fiber of it are isomorphic to those of a 
bouquet of $S^1$ and some spheres. In the course of the proof, 
the theory of the nearby and vanishing cycle sheaves 
for meromorphic functions will play an essential role. 
\end{abstract}

\maketitle

\section{Introduction}\label{section 1}

In \cite{Milnor}, for a complex hypersurface of 
$X= \CC^n$ Milnor proved that its 
Milnor fiber at an isolated singular point of it is 
homotopy euivalent to the bouquet 
$S^{n-1}\vee\cdots\vee S^{n-1}$ of some 
$(n-1)$-dimensional spheres $S^{n-1}$. 
After the discovery of this beautiful result, many 
mathematicians in different fields, such as 
singularity theory, algebraic geometry, topology, 
complex analysis, number theory and D-modules, 
studied them from various points of view. 
Globalizing the situation, in \cite{Broughton}, 
\cite{Paru} and \cite{S-T-1} for a complex 
polynomial function $f:X = \CC^n \longrightarrow \CC$ 
on $X= \CC^n$ satisfying some tameness condition 
at infinity, Broughton, Parusinski, Siersma and 
Tib{\u a}r proved that any fiber 
$f^{-1}(c) \subset X= \CC^n$ ($c \in \CC$) of it 
is homotopy equivalent to the bouquet of some 
$(n-1)$-dimensional spheres $S^{n-1}$: 
\begin{equation}
f^{-1}(c) \sim  S^{n-1}\vee\cdots\vee S^{n-1}.
\end{equation}
Since then several basic properties of polynomial functions 
were studied by a lot of mathematicians in various 
situations and the results thus obtained are surprisingly similar 
to the ones obtained in the local case. See for example 
\cite{ALM-1}, \cite{C-R}, \cite{Dimca2}, 
\cite{Dimca}, \cite{D-N}, \cite{D-S-1}, 
\cite{Douai}, \cite{E-T}, \cite{L-N-2}, \cite{G-L-M}, 
\cite{GZ-L-MH-new}, \cite{Ka-St-2}, \cite{L-S}, 
\cite{M-T-new3}, \cite{M-T-3}, \cite{M-T-2}, \cite{M-M}, \cite{N-Z-2}, 
\cite{N-N}, \cite{Raibaut-1.5}, \cite{Raibaut-1.8}, 
\cite{Sabbah}, \cite{Sabbah-2}, \cite{S-T}, 
\cite{Stapledon}, \cite{Tak-3}, \cite{T-T}, 
\cite{Tibar-book}. 

\medskip \par
The aim of this paper is to show some evidence for 
the possibility of extending the above mentioned fundamental 
result of \cite{Broughton}, \cite{Paru} and \cite{S-T-1} 
to rational functions on $X= \CC^n$ satisfying 
some tameness condition at infinity and the one at their 
points of indeterminacy. Specifically, assuming some 
conditions on a rational function 
\begin{equation}
f(z)= \frac{P(z)}{Q(z)}  \qquad (x \in X \setminus Q^{-1}(0)). 
\end{equation}
($P(z), Q(z) \in \CC [z_1, \ldots, z_n]$) on $X= \CC^n_z$ 
(see Section \ref{section 3} for the details),  
we prove that the cohomology groups of its 
generic fiber $f^{-1}(c) \subset X \setminus Q^{-1}(0)$ 
of $f: X \setminus Q^{-1}(0) \longrightarrow \CC$ 
are isomorphic to those of the bouquet 
$S^1 \vee (S^{n-1}\vee\cdots\vee S^{n-1})$ of $S^1$ and some 
$(n-1)$-dimensional spheres $S^{n-1}$. For this purpose, 
as well as some basic results on perverse sheaves, 
the theory of the nearby and vanishing cycle sheaves 
for meromorphic functions developed in 
\cite{N-T}, \cite{Raibaut-2} and \cite{Tak-4} are 
effectively used. Motivated by this cohomological 
result for the generic fiber $f^{-1}(c) \subset X \setminus Q^{-1}(0)$ 
of $f: X \setminus Q^{-1}(0) \longrightarrow \CC$, 
we then conjecture that it is also 
homotopy equivalent to $S^1 \vee (S^{n-1}\vee\cdots\vee S^{n-1})$: 
\begin{equation}
f^{-1}(c) \sim  S^1 \vee (S^{n-1}\vee\cdots\vee S^{n-1}).
\end{equation}
Moreover, a formula for the number of 
the spheres $S^{n-1}$ in this (conjectural)  
bouquet decomposition of $f^{-1}(c)$ will be given. 
By Lemma \ref{vc-lem} and \eqref{eff-tool}, we see also 
that our meromorphic vanishing cycle functor 
with compact support is a powerful tool to detect 
the change of the topologies of the fibers of $f$ 
at its points of indeterminacy. Namely we can study 
the ``singularities at indeterminacy" of $f$ by it. 
See Section \ref{section 3} for the details. 
In \cite{GZ-L-MH} Gusein-Zade, Luengo and Melle-Hern\'andez 
generalized Milnor's fibration theorem in \cite{Milnor} 
to meromorphic functions on complex manifolds 
and defined their Milnor fibers. Moreover they obtained 
a formula for their monodromy zeta functions. 
Since then many mathematicians studied 
basic properties of rational and meromorphic functions 
in various situations 
(see e.g. \cite{A-G-L-N}, \cite{B-P}, \cite{B-P-S}, \cite{GZ-L-MH-new-new}, 
\cite{Thang}, \cite{N-S-T}, \cite{N-T}, \cite{Raibaut-2}, \cite{S-T-2}, \cite{Tak-2}, 
\cite{Tibar}). However, in contrast to 
Milnor fibers of holomorphic functions, 
the geometric structures of those of meromorphic functions 
are not fully understood yet. In view of our results in this 
paper and \cite[Theorem 1.2]{N-T}, one may imagine that 
also the Milnor fiber of the meromorphic function in 
\cite[Theorem 1.2]{N-T} has the same 
homotopy type $S^1 \vee (S^{n-1}\vee\cdots\vee S^{n-1})$. 
It would be a challenging problem to prove it or 
give some evidence for it. 

\bigskip
\noindent{\bf Acknowledgement}: The author 
is very grateful to Professor Adam Parusinski and 
Ren Fernandes for some stimulating discussions with them 
during the preparation of this paper.

\section{Preliminary notions and results}\label{section 2}
In this section, we introduce preliminary notions and results 
which will be used in this paper. 
We essentially follow the terminology of \cite{Dimca}, \cite{H-T-T}, 
\cite{K-S} and \cite{Sch}. 
For example, for a topological space $X$ we denote by $\Db(X)$ the derived category whose 
objects are bounded complexes of sheaves of $\CC_X$-modules on $X$. For a complex 
analytic space $X$, let $\Dbc(X)$ be the full subcategory of $\Db(X)$ consisting 
of constructible complexes of sheaves. Recall that 
for a complex manifold $X$ of dimension $n$ we have 
the (generalized) Poincar\'e duality isomorphisms 
\begin{equation}
H^j(X; \CC_X) \simeq [ H^{2n-j}_c(X; \CC_X)]^* 
\quad (j \in \ZZ ). 
\end{equation}
In this paper, we employ the convention 
that for such $X$ the shifted constant sheaf $\CC_X [n] \in \Dbc (X)$ 
is a perverse sheaf on $X$. In what follows, let 
$X$ be a complex manifold of dimension $n$ 
and $f\colon X \longrightarrow \CC$ 
a non-constant holomorphic function on it. Set 
$X_0:=f^{-1}(0)\subsetneqq X$. Then for any point 
$x \in X_0=f^{-1}(0)\subset X$ of 
$X_0$ we can define the Milnor fiber 
$F_x \subset X \setminus X_0$ of $f$ at it  
(see \cite{Milnor}). 
The following beautiful theorem was proved by Milnor in 
\cite{Milnor} 

\begin{theorem}\label{thm.7.2.003}
{\rm (Milnor \cite{Milnor})} 
In the situation as above, assume that 
the complex hypersurface $X_0=f^{-1}(0)\subset X$ 
has an isolated singular point at a point 
$x \in X_0=f^{-1}(0)\subset X$. 
Then the Milnor fiber $F_x \subset X \setminus X_0$ of $f$ at it 
is homotopy equivalent to the bouquet of some 
$(n-1)$-dimensional spheres $S^{n-1}$: 
\begin{equation}
F_x  \sim  S^{n-1}\vee\cdots\vee S^{n-1}.
\end{equation}
In particular, for its cohomology groups 
$H^j(F_x; \CC )$ ($j \in \ZZ$) we have 
\begin{equation}
H^j(F_x; \CC ) \simeq 
\begin{cases}
\ \CC & (j=0) \\
& \\
\ \CC^{\mu} & (j=n-1) \\
 & \\
\ 0 & (j \neq 0,n-1), 
\end{cases}
\end{equation}
where $\mu>0$ is the number of spheres $S^{n-1}$ in 
the above bouquet decomposition. 
\end{theorem}
We call the number $\mu >0$ in Theorem \ref{thm.7.2.003} 
the Milnor number of $f$ (or of $X_0=f^{-1}(0)$) 
at the point $x \in X_0=f^{-1}(0)\subset X$. 
Note that the Milnor fiber $F_x \subset X 
\setminus X_0$ in 
Theorem \ref{thm.7.2.003} is an  
$(n-1)$-dimensional complex manifold i.e. a 
$(2n-2)$-dimensional real analytic manifold. 
Theorem \ref{thm.7.2.003} asserts that it 
has the homotopy type of the real 
$(n-1)$-dimensional topological space 
$S^{n-1}\vee\cdots\vee S^{n-1}$. 
From now, we shall explain how 
Milnor fibers and their monodromies can be effectively 
studied by the more sophisticated machineries 
of constructible sheaves. 
Let $i\colon X_0 = f^{-1}(0) \hookrightarrow X$
and $j\colon X\setminus X_0 \hookrightarrow X$ be 
the inclusion maps. For $\CC^* := \CC \setminus\{0\}$ 
we regard the exponential map 
\begin{equation}
\exp : \widetilde{\CC^*} := \CC  \longrightarrow \CC^*, 
\qquad (t \longmapsto \exp (t)) 
\end{equation}
as its universal covering. Then for 
the fiber product $\widetilde{X\setminus X_0}$ of 
$\exp : \widetilde{\CC^*} = \CC  \longrightarrow \CC^*$ and 
$f|_{X\setminus X_0}\colon X\setminus X_0 \longrightarrow \CC^*$
we obtain the following commutative diagram. 

\begin{equation}\begin{xy}\xymatrix@C=30pt{
& & \,\widetilde{X\setminus X_0}\, \ar[ld]_-{j\circ\pi} \ar[d]_-{\pi} \ar[r] 
\ar@{}[dr]|{\square} & \,\widetilde{\CC^*} \ar[d]^-{\mathrm{exp}}
 & \hspace{-48pt} {\,\simeq\,\CC} & \\
X_0=f^{-1}(0)\,\,\, \ar@{^{(}->}@<-0.3ex>[r]_-{i} & \,X\, & \,X\setminus X_0\, 
\ar@{_{(}->}@<+0.3ex>[l]^-{j} \ar[r]_-{f|_{X\setminus X_0}} & \,\CC^*. &
}\end{xy}\end{equation}
Here $\Box$ stands for the Cartesian diagram. 

\begin{definition}\label{def.7.1.001}
{\rm (Deligne~\cite{Deligne})} 
For $\F \in \Db (X)$ we define its nearby cycle 
complex $\psi_f( \F )\in \Db (X_0)$ along $f$ by 
\begin{equation}
\psi_f( \F ) = i^{-1} {\rm R} (j\circ\pi)_*(j\circ\pi)^{-1} \F
\ \in \Db (X_0).
\end{equation}
Moreover, we define the vanishing cycle 
complex $\phi_f( \F )\in \Db (X_0)$ of $\F \in \Db (X)$ 
along $f$ so that it fits into the distinguished triangle 
\begin{equation}
i^{-1} \F \longrightarrow \psi_f( \F ) \longrightarrow 
\phi_f( \F ) \overset{+1}{\longrightarrow}, 
\end{equation}
where the morphism $i^{-1} \F  \longrightarrow \psi_f( \F )$ 
is induced by the one 
${\rm id} \longrightarrow {\rm R} (j\circ\pi)_*(j\circ\pi)^{-1}$ 
of functors (for the more precise definition, see 
Kashiwara-Schapira \cite[page 351]{K-S}). 
\end{definition}
We thus obtain two functors 
\begin{equation}
\psi_f( \cdot ), \ \phi_f( \cdot ) \colon \Db (X) \longrightarrow \Db (X_0). 
\end{equation}
We call them the nearby and vanishing cycle functors along $f$ 
respectively. Moreover for $\F \in \Db (X)$ 
there exist its monodromy automorphisms 
\begin{equation}
\Psi_f( \F ) \colon \psi_f( \F ) \simto \psi_f( \F ), \qquad 
\Phi_f( \F ) \colon \phi_f( \F ) \simto \phi_f( \F ). 
\end{equation} 
Now we recall the following basic result. 

\begin{proposition}\label{prp.7.1.002}
{\rm (see e.g. \cite[Exercise VIII.15]{K-S} and 
\cite[Proposition 4.2.11]{Dimca})}  
Let $\rho \colon Y  \longrightarrow X$ be a 
morphism of complex manifolds 
and $f\colon X  \longrightarrow \CC$ a non-constant 
holomorphic function on $X$. We set 
$g:=f\circ\rho\colon Y  \longrightarrow \CC$ and 
\begin{equation}
X_0=f^{-1}(0)\subset X,\quad Y_0=g^{-1}(0)=\rho^{-1}(X_0)\subset Y. 
\end{equation}
Let $\rho|_{Y_0}\colon Y_0 \longrightarrow X_0$ be the 
restriction of $\rho$ to $Y_0\subset Y$ and assume that 
$\rho \colon Y  \longrightarrow X$ is proper on 
the support of $\G \in \Db (Y)$. Then there exist isomorphisms 
\begin{equation}
\psi_f( {\rm R} \rho_* \G ) \simeq {\rm R} (\rho|_{Y_0})_*\psi_g( \G ), \qquad 
\phi_f( {\rm R} \rho_* \G ) \simeq {\rm R} (\rho|_{Y_0})_*\phi_g( \G ). 
\end{equation} 
\end{proposition}
By the following basic results we can study 
Milnor fibers and their monodromies via nearby and 
vanishing cycle sheaves. 

\begin{theorem}\label{thm.7.2.004}
{\rm (see e.g. \cite[Proposition 4.2.2]{Dimca} 
and \cite[Theorem 2.6]{Tak-3})} 
For any point 
$x \in X_0$ of $X_0=f^{-1}(0)$ there exist isomorphisms 
\begin{equation}
H^j \psi_f( \CC_X )_x \simeq H^j(F_x; \CC ),
\quad 
H^j \phi_f(  \CC_X )_x \simeq \widetilde{H}^j(F_x; \CC )
\qquad (j\in\ZZ), 
\end{equation}
where $\widetilde{H}^j(F_x; \CC )$ stand for 
the reduced cohomology groups of $F_x$. 
Moreover these isomorphisms are compatible 
with the monodromy automorphisms of the both sides. 
\end{theorem}
More generally, we have the following result. 

\begin{theorem}\label{thm.7.2.006}
For any constructible sheaf $\F \in \Dbc (X)$ on $X$ and 
any point $x \in X_0$ of $X_0=f^{-1}(0)$ there exist isomorphisms 
\begin{equation}
 H^j \psi_f( \F )_x \simeq H^j(F_x; \F )
 \qquad (j\in\ZZ).
\end{equation}
\end{theorem}

\begin{theorem}\label{thm.7.1.004}
{\rm (see e.g. \cite[Proposition 8.6.3]{K-S}, 
\cite[page 103]{Dimca} and \cite[Theorem 4.0.2 and Lemma 4.2.1]{Sch})} 
Assume that $\F \in \Db(X)$ is constructible. Then 
$\psi_f( \F ), \phi_f( \F ) \in  \Db(X_0)$ 
are also constructible. Namely there exist two  
functors 
\begin{equation}
\psi_f(  \cdot  ),\ 
\phi_f(  \cdot  ) \colon \Dbc (X) \longrightarrow \Dbc (X_0).
\end{equation}
\end{theorem}
Shifting the functors $\psi_f( \cdot ), \phi_f( \cdot )$ by 
$-1 \in \ZZ$ we set 
\begin{equation}
^p\!\psi_f( \cdot ) := \psi_f( \cdot )[-1], \quad 
^p\!\phi_f( \cdot ) := \phi_f( \cdot )[-1].
\end{equation}
Then we have the following result. 

\begin{theorem}\label{thm.7.2.007}
Assume that $\F \in \Dbc (X)$ is a perverse 
sheaf on $X$. Then 
$^p\!\psi_f( \F ), \ ^p\!\phi_f( \F )\in \Dbc (X_0)$ are 
also perverse sheaves on $X_0=f^{-1}(0)$. 
\end{theorem}

Next, we shall introduce some basic definitions 
and results on polynomial maps 
$f \colon X=\CC^n \longrightarrow \CC$ on $X= \CC^n_z$ ($n \geq 2$). 
After the fundamental papers Broughton 
\cite{Broughton}, Parusinski \cite{Paru} 
and Siersma-Tib{\u a}r \cite{S-T-1}, many authors 
studied the global behavior of polynomial maps 
$f \colon X= \CC^n \longrightarrow \CC$. 
First of all, for a 
polynomial map $f \colon X= \CC^n \longrightarrow \CC$, 
it is well-known that there 
exists a finite subset $B \subset \CC$ such that the restriction
\begin{equation}
\CC^n \setminus f^{-1}(B) \longrightarrow \CC \setminus B
\end{equation}
of $f$ is a $C^{\infty}$ locally trivial fibration. 
We denote by $B_f$ the smallest 
subset $B \subset \CC$ satisfying this 
condition. Let ${\rm Sing} f \subset X= \CC^n$ 
be the set of the critical points of 
$f \colon X= \CC^n \longrightarrow \CC$. Then by 
the definition of $B_f$, obviously we have 
$f( {\rm Sing} f) \subset B_f$. 
We call the elements of $B_f$ the 
bifurcation values of $f$. 
The determination of the bifurcation set 
$B_f \subset \CC$ is a fundamental problem and has been  
studied by many mathematicians from 
various points of view. 
For $r>0$ we set $B(0;r) := \{ z \in X= \CC^n \ | \ 
 ||z|| = \sqrt{ \sum_{i=1}^n |z_i|^2} <r \} \subset X= \CC^n$. 

\begin{definition} {\rm (Broughton \cite{Broughton})} \label{dfn:tame}
Let $\partial f\colon X= \CC^n \longrightarrow \CC^n$ 
be the map defined by $\partial f(z)=(\partial_1f(z), \ldots, \partial_n f(z))$ 
($z \in \CC^n$). 
Then we say that $f$ is tame at infinity if the restriction $(\partial f)^{-1}(B(0;\e )) 
\longrightarrow B(0;\e )$ of $\partial f$ to a sufficiently 
small ball $B(0;\e ) \subset \CC^n$ centered at the origin $0 \in \CC^n$ is proper. 
\end{definition}
Let $Y= \CC^n_w$ be the dual vector space of $X= \CC^n_z$ and 
for $z=(z_1, \ldots, z_n) \in X$ and 
$w=(w_1, \ldots, w_n) \in Y$ set 
\begin{equation}
\langle z, w \rangle := \sum_{i=1}^n z_i w_i \qquad \in \CC. 
\end{equation}
Let
\begin{equation}
X\overset{p}{\longleftarrow}X\times 
Y\overset{q}{\longrightarrow}Y
\end{equation}
be the projections. Then, the contangent 
bundle $T^*X$ of $X= \CC^n$ being naturally 
identified with $X \times Y$, we define a 
Lagrangian submanifold $\Lambda^f$ of 
$T^*X \simeq X \times Y$ by 
\begin{equation}
\Lambda^f := \{ (z,f(z)) \in T^*X \ | \ z \in X \} 
\quad \subset T^*X \simeq X \times Y. 
\end{equation}
By the condition $\dim \Lambda^f = \dim Y$ there 
exists a non-empty Zariski open subset $\Omega \subset Y$ 
of $Y= \CC^n_w$ such that the restriction 
\begin{equation}
q^{-1}( \Omega ) \cap \Lambda^f \longrightarrow 
\Omega 
\end{equation}
of the projection $q: X \times Y  \longrightarrow Y$ 
to $q^{-1}( \Omega ) \cap \Lambda^f \subset X \times Y$ 
is an unramified covering of finite degree 
(see e.g. \cite[page 43]{GM}). It follows that if the origin 
$0 \in Y= \CC^n$ of $Y= \CC^n$ is contained 
in $\Omega$ then 
$f \colon X= \CC^n \longrightarrow \CC$ is tame at 
infinity. If the origin 
$0 \in Y= \CC^n$ is not contained 
in $\Omega$ , by taking a point $w_0 \in \Omega 
\subset Y= \CC^n$ of $\Omega$ 
close to the origin $0 \in Y= \CC^n$ 
we define a new function 
$g \colon X= \CC^n \longrightarrow \CC$ by 
\begin{equation}
g(z):= f(z)- \langle z, w_0 \rangle \quad (z \in X= \CC^n). 
\end{equation}
Then it is tame at infinity. Namely, generic linear 
perturbations of $f \colon X= \CC^n \longrightarrow \CC$ 
are tame at infinity. 
In \cite{Broughton} Broughton showed that if 
$f \colon X= \CC^n \longrightarrow \CC$ is tame at infinity 
then we have $f( {\rm Sing} f) = B_f$. 
In the tame case, the following result is fundamental. 
Note that if $f \colon X= \CC^n \longrightarrow \CC$ is tame at 
infinity then ${\rm Sing} f \subset X= \CC^n$ is a 
finite subset. This in particular implies that for any 
point $c \in \CC$ its fiber $f^{-1}(c) \subset X= \CC^n$ 
of $f$ has only isolated singular points. We denote by 
$\mu \geq 0$ the sum of the Milnor numbers of $f$ 
at the points of ${\rm Sing} f$ and call it the total 
Milnor number of $f$. Moreover, for a point $c \in \CC$ 
we denote by $\mu_c \geq 0$ the sum of the Milnor numbers of $f$  
at the (isolated) singular points of $f^{-1}(c)$. 

\begin{theorem} {\rm (Broughton \cite{Broughton})} \label{tame}
Assume that $f$ is tame at infinity. 
Then for any point $c \in \CC$ its fiber $f^{-1}(c)$ of $f$ 
is homotopy equivalent to the bouquet of some 
$(n-1)$-dimensional spheres $S^{n-1}$: 
\begin{equation}
f^{-1}(c)  \sim  S^{n-1}\vee\cdots\vee S^{n-1}.
\end{equation}
Moreover the number of spheres $S^{n-1}$ in this bouquet 
decomposition is equal to $\mu - \mu_c \geq 0$. 
In particular, we have isomorphisms 
\begin{equation}
H^j(f^{-1}(c); \CC ) \simeq 
\begin{cases}
\ \CC & (j=0) \\
& \\
\ \CC^{\mu - \mu_c} & (j=n-1) \\
 & \\
\ 0 & (j \neq 0,n-1). 
\end{cases}
\end{equation}
\end{theorem}
By Parusinski \cite{Paru} and Siersma-Tib\u ar \cite{S-T-1} 
this fundamental result 
was later extended to polynomial maps 
$f \colon X= \CC^n \longrightarrow \CC$ 
with isolated singularities 
at infinity. 
In what follows, assume that $f: X = \CC^n \longrightarrow \CC$ 
is tame at infinity and set $\Sigma := f(  {\rm Sing } f) =B_f  
\subset \CC$. For $c \in \CC \setminus \Sigma$ we call 
$f^{-1}(c) \subset X= \CC$ a generic fiber of 
$f: X = \CC^n \longrightarrow \CC$. 
Then in \cite{L-T} Libgober and Tib{\u a}r obtained the following 
remarkable result. 

\begin{proposition}\label{prp.L-T}
{\rm (Libgober-Tib{\u a}r \cite[Proposition 3.1]{L-T})}  
Asuume that $c \in \CC \setminus \Sigma$. Then 
the hypersurface complement 
$X \setminus f^{-1}(c)$ has the homotopy type of 
\begin{equation}
S^1 \vee (S^{n} \vee \cdots \vee S^{n}), 
\end{equation}
where $S^{n} \vee \cdots \vee S^{n}$ is 
the bouquet of the $\mu$ copies of the $n$-sphere $S^{n}$.  
\end{proposition}
If $c \in \CC \setminus \Sigma$, this in particular implies that 
for the cohomology groups 
$H^j(X \setminus f^{-1}(c); \CC )$ ($j \in \ZZ$) we have isomorphisms 
\begin{equation}
H^j(X \setminus f^{-1}(c); \CC ) \simeq 
\begin{cases}
\ \CC & (j=0,1) \\
& \\
\ \CC^{\mu} & (j=n) \\
 & \\
\ 0 & (j \neq 0,1,n), 
\end{cases}
\end{equation}
We can easily recover this result on the cohomology groups 
by using Theorem \ref{tame}, the distinguished triangle 
\begin{equation}
{\rm R}\Gamma_c(X; \CC_{X \setminus f^{-1}(c)}) \longrightarrow 
{\rm R}\Gamma_c(X; \CC_X) \longrightarrow 
{\rm R}\Gamma_c(X; \CC_{f^{-1}(c)}) \overset{+1}{\longrightarrow}
\end{equation}
and the (generalized) Poincar\'e dualities for the complex 
manifolds $X \setminus f^{-1}(c)$ and $f^{-1}(c)$. 
For $c \in \Sigma$ we have the following result. 

\begin{proposition}\label{prp.gen-L-T} 
Asuume that $n \geq 3$ and $c \in \Sigma$. Then for the complement 
$X \setminus f^{-1}(c)$ of the singular fiber $f^{-1}(c)$ of 
$f$ we have 
\begin{equation}
H^j(X \setminus f^{-1}(c); \CC ) \simeq 0 \quad 
(j \not= 0,1,n-1,n) 
\end{equation}
and $H^0(X \setminus f^{-1}(c); \CC ) \simeq 
H^1(X \setminus f^{-1}(c); \CC ) \simeq \CC$. Moreover we have 
\begin{equation}
\dim H^n(X \setminus f^{-1}(c); \CC ) - 
\dim H^{n-1}(X \setminus f^{-1}(c); \CC )= 
\mu - \mu_c.
\end{equation}
\end{proposition} 

\begin{proof}
Let $i_c: \{ c \} \hookrightarrow \CC$ be the inclusion map 
and consider the distinguished triangle 
\begin{equation}\label{d-t-1} 
i_c^{-1} {\rm R}f_! \CC_X \longrightarrow \psi_{t-c} {\rm R}f_! \CC_X 
\longrightarrow 
\phi_{t-c} {\rm R}f_! \CC_X \overset{+1}{\longrightarrow}. 
\end{equation}
Note that by Theorem \ref{thm.7.2.006} there exist isomorphisms 
\begin{align} 
i_c^{-1} {\rm R}f_! \CC_X & \simeq {\rm R}\Gamma_c 
(f^{-1}(c); \CC_{f^{-1}(c)}), \\ 
\psi_{t-c} {\rm R}f_! \CC_X & \simeq {\rm R}\Gamma_c 
(f^{-1}(c+ \eta ); \CC_{f^{-1}(c+ \eta)}) 
\quad (0< \eta \ll 1). 
\end{align}
Since for $0< \eta \ll 1$ the fiber $f^{-1}(c+ \eta )$ of $f$ is 
generic, by Theorem \ref{tame} we thus obtain 
isomorphisms 
\begin{equation}
H^j \psi_{t-c} {\rm R}f_! \CC_X  \simeq 
\begin{cases}
\ \CC & (j=2n-2) \\
& \\
\ \CC^{\mu} & (j=n-1) \\
 & \\
\ 0 & (j \neq n-1, 2n-2). 
\end{cases}
\end{equation}
By the tameness at infinity of $f$, 
there exist $r_0 \gg 0$ and a small disk 
$D_{\eta_0}(c) := \{ t \in \CC \ | \ 
 |t-c| < \eta_0 \} \subset \CC$ ($0< \eta_0 \ll 1$) 
in $\CC$ centered at the point $c \in \CC$ such that 
for any $r \geq r_0$ and $t \in D_{\eta_0}(c)$ 
the sphere $\partial B(0;r)$ intersects $f^{-1}(t)$ 
transversally (see e.g. \cite[Proposition 3.4]{Tak-4}). 
By the 
non-characteristic deformation lemma \cite[Proposition 2.7.2]{K-S}, 
this implies that we have isomorphisms 
\begin{align} 
i_c^{-1} {\rm R}f_! \CC_{B(0;r_0)} & \simeq 
i_c^{-1} {\rm R}f_! \CC_X, \\ 
\psi_{t-c} {\rm R}f_! \CC_{B(0;r_0)} & \simeq 
\psi_{t-c} {\rm R}f_! \CC_X
\end{align}
and hence the distinguished triangle 
\begin{equation}\label{d-t-2} 
i_c^{-1} {\rm R}f_! \CC_{B(0;r_0)} \longrightarrow 
\psi_{t-c} {\rm R}f_! \CC_{B(0;r_0)} \longrightarrow 
\phi_{t-c} {\rm R}f_! \CC_{B(0;r_0)} \overset{+1}{\longrightarrow}. 
\end{equation}
is isomorphic to the one \eqref{d-t-1}. Moreover, by 
Proposition \ref{prp.7.1.002}
there exists an isomorphism 
\begin{equation} 
\phi_{t-c} {\rm R}f_! \CC_{B(0;r_0)} \simeq 
{\rm R}\Gamma (f^{-1}(c); \phi_{f-c} \CC_{B(0;r_0)}).  
\end{equation}
We can easily see that the support of the vanishing 
cycle sheaf $\phi_{f-c} \CC_{B(0;r_0)}$ is contained 
in the finite set of the (isolated) singular points 
of the complex hypersurface $f^{-1}(c) \subset X= \CC^n$. 
Then we obtain a concentration 
\begin{equation}
H^j\phi_{t-c} {\rm R}f_! \CC_{B(0;r_0)} \simeq 0 
\quad (j \not= n-1) 
\end{equation}
and $\dim H^{n-1}\phi_{t-c} {\rm R}f_! \CC_{B(0;r_0)}
= \mu_c$. By the cohomology long exact sequence 
associated to the distinguished triangle \eqref{d-t-1}, 
we thus can show that 
\begin{equation}
H^j_c(f^{-1}(c); \CC ) \simeq 0 \quad (j \not= n-1, n, 2n-2), 
\quad H^{2n-2}_c(f^{-1}(c); \CC ) \simeq \CC
\end{equation}
and 
\begin{equation}
\dim H^{n-1}_c(f^{-1}(c); \CC ) - \dim H^{n}_c(f^{-1}(c); \CC )
= \mu - \mu_c. 
\end{equation}
Then the assertion immediately follows from the 
distinguished triangle 
\begin{equation}
{\rm R}\Gamma_c(X; \CC_{X \setminus f^{-1}(c)}) \longrightarrow {\rm R}\Gamma_c(X; \CC_X) 
\longrightarrow 
{\rm R}\Gamma_c(X; \CC_{f^{-1}(c)}) \overset{+1}{\longrightarrow}
\end{equation}
and the (generalized) Poincar\'e duality for the complex 
manifold $X \setminus f^{-1}(c)$. 
\qed
\end{proof} 

Finally, we recall the theory of Milnor fibers of 
meromorphic functions introduced by 
Gusein-Zade, Luengo and Melle-Hern\'andez in \cite{GZ-L-MH} 
and the nearby and vanishing cycle functors related to 
them in \cite{N-T}, \cite{Raibaut-2} and \cite{Tak-4}. 
Let $X$ be a connected complex manifold and $P(x), Q(x)$ 
holomorphic functions on it. Assume that $Q(x)$ 
is not identically zero on $X$. 
Then we define a meromorphic function $f(x)$ on 
$X$ by 
\begin{equation}
f(x)= \frac{P(x)}{Q(x)}  \qquad (x \in X \setminus Q^{-1}(0)). 
\end{equation}
Let us set $I(f)=P^{-1}(0) \cap Q^{-1}(0) \subset X$. 
If $P$ and $Q$ are coprime in the local ring 
$\O_{X,x}$ at a point $x \in X$, then 
$I(f)$ is nothing but the set of the indeterminacy 
points of $f$ on a neighborhood of $x$. In fact, 
the set $I(f)$ depends on the pair $(P(x), Q(x))$ of 
holomorphic functions representing $f(x)$. 
For example, if we take a holomorphic function 
$R(x)$ on $X$ (which is not identically zero on $X$) and set 
\begin{equation}
g(x)= \frac{P(x)R(x)}{Q(x)R(x)}  \qquad (x \in X), 
\end{equation}
then the set $I(g)=I(f) \cup R^{-1}(0)$ might be 
bigger than $I(f)$. In this way, we 
distinguish $f(x)= \frac{P(x)}{Q(x)}$ from 
$g(x)= \frac{P(x)R(x)}{Q(x)R(x)}$ even if 
their values coincide over an open dense 
subset of $X$. 
This is the convention due to 
Gusein-Zade, Luengo and Melle-Hern\'andez 
\cite{GZ-L-MH} etc. 
Then we have the following 
fundamental theorem due to \cite{GZ-L-MH}. 

\begin{theorem}\label{the-fib} 
{\rm (Gusein-Zade, Luengo and Melle-Hern\'andez 
\cite[Section 1]{GZ-L-MH})}  
For any point $x \in P^{-1}(0)$ there exists 
$\e_0> 0$ such that for any $0< \e \leq \e_0$ 
and the open ball $B(x; \e ) \subset X$ of 
radius $\e >0$ centered at $x$ 
(in a local chart of $X$) the restriction 
\begin{equation}
B(x; \e ) \setminus Q^{-1}(0) 
\longrightarrow \CC
\end{equation}
of $f: X \setminus Q^{-1}(0) 
\longrightarrow \CC$ is a locally trivial 
fibration over a sufficiently small 
punctured disk in $\CC$ centered at 
the origin $0 \in \CC$ 
\end{theorem}

We call the fiber in this theorem the Milnor fiber of 
the meromorphic function $f(x)= \frac{P(x)}{Q(x)}$ 
at $x \in P^{-1}(0)$ and denote it by $F_x$. 
Then we can extend the classical 
notions of nearby and vanishing cycle functors to those for 
meromorphic functions as follows 
(for the details, see \cite{N-T}, 
\cite{Raibaut-2} and \cite{Tak-4}). For the function $f= \frac{P}{Q}: 
X \setminus Q^{-1}(0) \longrightarrow \CC$ 
we consider the (not necessarily) closed embedding 
\begin{equation}
i_f: X \setminus Q^{-1}(0) \hookrightarrow 
X \times \CC_t \qquad (x \longmapsto (x,f(x))).  
\end{equation}
Let $t: X \times \CC \rightarrow \CC$ be the 
second projection. 

\begin{definition}
{\rm (Nguyen-Takeuchi \cite{N-T})} For 
$\F \in \Db(X)$ we set 
\begin{equation}
\psi_f^{\mer}( \F ):= \psi_t( {\rm R} i_{f*}
( \F |_{X \setminus Q^{-1}(0)}) ) 
\qquad \in \Db(X_0). 
\end{equation}
We call $\psi_f^{\mer}( \F )$ the meromorphic 
nearby cycle complex of $\F$ along $f$. 
\end{definition} 
We thus obtain a functor 
\begin{equation}
\psi_f^{\mer}(  \cdot  ) : 
\Db(X) \longrightarrow \Db(X). 
\end{equation}
We call it the meromorphic 
nearby cycle functor along $f$. 
As in the case where $f$ is holomorphic i.e. 
$f=\frac{P}{1}$, we have the following results. 

\begin{lemma}\label{lem-nc-1} 
{\rm (see \cite[Lemma 2.1]{N-T})} 
\begin{enumerate}
\item[\rm{(i)}] The support of $\psi_f^{\mer}( \F )$ is 
contained in $P^{-1}(0)$. 
\item[\rm{(ii)}] There exists an isomorphism 
\begin{equation}
\psi_f^{\mer}( \F ) \simto 
\psi_f^{\mer}( {\rm R} 
\Gamma_{X \setminus (P^{-1}(0) \cup Q^{-1}(0))} 
( \F )). 
\end{equation}
\end{enumerate}
\end{lemma}

\begin{theorem}\label{the-nc-1} 
{\rm (see \cite[Lemma 2.1 and Theorem 2.2]{N-T})} 
Assume that $\F \in \Db(X)$ is constructible. Then we have 
\begin{enumerate}
\item[\rm{(i)}] 
$\psi_f^{\mer}( \F ) \in \Db(X)$ is also constructible. 
\item[\rm{(ii)}] If moreover $\F \in \Db(X)$ is perverse, then 
$\psi_f^{\mer}( \F )[-1] \in \Db(X)$ is also perverse. 
\item[\rm{(iii)}] For any point $x \in P^{-1}(0)$ 
there exist isomorphisms 
\begin{equation}
H^j \psi_f^{\mer}( \F )_x \simeq H^j (F_x; \F ) 
 \qquad (j \in \ZZ ). 
\end{equation}
Moreover these isomorphisms are 
compatible with the monodromy automorphisms 
on the both sides. 
\end{enumerate}
\end{theorem}
By this theorem we obtain a functor 
\begin{equation}
\psi_f^{\mer}(  \cdot ) : 
\Dbc (X) \longrightarrow \Dbc (X). 
\end{equation}

For the meromorphic function $f(x)= \frac{P(x)}{Q(x)}$ 
and $\F \in \Db(X)$ we set also 
\begin{equation}
\psi_f^{\merc}( \F ):= \psi_t( i_{f!}
( \F |_{X \setminus Q^{-1}(0)}) ) 
\qquad \in \Db(X). 
\end{equation}
In fact, this object was originally defined by 
Raibaut in \cite{Raibaut-2}. 
Here we call it the meromorphic 
nearby cycle complex 
with compact support of $\F$ along $f$. 
Then we obtain a functor 
\begin{equation}
\psi_f^{\merc}( \cdot ) : 
\Db(X) \longrightarrow \Db(X)
\end{equation}
which satisfies the properties similar to 
the ones in  Lemma \ref{lem-nc-1} and 
Theorem \ref{the-nc-1} (i), (ii) 
as follows (see \cite[Section 2]{Tak-3}). 
However the isomorphism in Theorem \ref{the-nc-1} (iii) 
does not hold for $\psi_f^{\merc}( \F )$. 
This implies that the natural morphism 
\begin{equation}
\psi_f^{\merc}( \F ) \longrightarrow 
\psi_f^{\mer}( \F )
\end{equation}
is not an isomorphism in general. 

\begin{lemma}\label{lem-1} 
{\rm (see \cite[Lemma 2.2]{Tak-4})}
\begin{enumerate}
\item[\rm{(i)}] The support of $\psi_f^{\merc}( \F )$ is 
contained in $P^{-1}(0)$. 
\item[\rm{(ii)}] There exists an isomorphism 
\begin{equation}
\psi_f^{\merc}( \F ) \simto 
\psi_f^{\merc}( \F_{X \setminus (P^{-1}(0) \cup Q^{-1}(0))} ). 
\end{equation}
\end{enumerate}
\end{lemma}

\begin{theorem}\label{the-1} 
{\rm (see \cite[Lemma 2.2 and Theorem 2.4]{Tak-4})}
Assume that $\F \in \Db(X)$ is constructible. Then we have 
\begin{enumerate}
\item[\rm{(i)}] 
$\psi_f^{\merc}( \F ) \in \Db(X)$ is also constructible. 
\item[\rm{(ii)}] If moreover $\F \in \Db(X)$ is perverse, then 
$\psi_f^{\merc}( \F )[-1] \in \Db(X)$ is also perverse. 
\item[\rm{(iii)}] 
For any point $x \in P^{-1}(0)$ 
there exist isomorphisms 
\begin{equation}
H^j \psi_f^{\merc}( \F )_{x} \simeq 
H^j_c ( \overline{F_{x}} \setminus Q^{-1}(0) ; \F )
 \qquad (j \in \ZZ ),  
\end{equation}
where $H^j_c ( \cdot )$ stands for the hypercohomology 
group with compact support. 
\end{enumerate}
\end{theorem}

By this theorem we obtain a functor 
\begin{equation}
\psi_f^{\merc }( \cdot ) : 
\Dbc (X) \longrightarrow \Dbc (X). 
\end{equation}
Let $i_0: X \hookrightarrow X \times \CC_{\tau}$ 
($x \longmapsto (x,0)$) be the inclusion map 
and for $\F \in \Db(X)$ set 
\begin{equation}
\phi_f^{\merc}( \F ):= \phi_{t}( i_{f!}
( \F |_{X \setminus Q^{-1}(0)}) ) 
\in \Db(X). 
\end{equation}
We call it the meromorphic 
vanishing cycle sheaf with compact support 
of $\F$ along $f$. Then there exists a 
distinguished triangle 
\begin{equation}\label{dist}
i_0^{-1} i_{f!}( \F |_{X \setminus Q^{-1}(0)}) 
\longrightarrow \psi_f^{\merc}( \F )
\longrightarrow \phi_f^{\merc}( \F ) 
\overset{+1}{\longrightarrow}
\end{equation}
in $\Db(X)$. Moreover, since for the inclusion map 
$j: P^{-1}(0) \setminus  Q^{-1}(0) 
\hookrightarrow X$ we have a Cartesian diagram 
\begin{equation}
\begin{CD}
P^{-1}(0) \setminus  Q^{-1}(0)   @>{j}>> X 
\\
@VVV    @VV{i_0}V
\\
X \setminus  Q^{-1}(0)   @>{i_f}>>  
X \times \CC_t
\end{CD}
\end{equation}
of inclusion maps, we obtain an isomorphism 
\begin{equation}\label{dist-isom}
i_0^{-1} i_{f!}( \F |_{X \setminus Q^{-1}(0)}) 
\simeq 
j_! ( \F |_{P^{-1}(0) \setminus  Q^{-1}(0) } ) 
\simeq \F_{P^{-1}(0) \setminus  Q^{-1}(0)}. 
\end{equation}
We then see that also the functor 
\begin{equation}
\phi_f^{\merc }( \cdot ) [-1] : \Db (X) \longrightarrow \Db (X)
\end{equation}
preserves the constructibility and the perversity. 
Moreover by \eqref{dist} and \eqref{dist-isom},  
for any $\F \in \Db(X)$ we have isomorphisms 
\begin{equation}
\phi_f^{\merc}( \F )_{x} \simeq 
\psi_f^{\merc}( \F )_{x} \qquad 
(x \in Q^{-1}(0)). 
\end{equation}
We also have the following result. 

\begin{lemma}\label{lem-2} 
{\rm (see \cite[Lemma 2.5]{Tak-4})}
\begin{enumerate}
\item[\rm{(i)}] The support of $\phi_f^{\merc}( \F )$ is 
contained in $P^{-1}(0)$. 
\item[\rm{(ii)}] There exists an isomorphism 
\begin{equation}
\phi_f^{\merc}( \F ) \simto 
\phi_f^{\merc}( \F_{X \setminus Q^{-1}(0)} ). 
\end{equation}
\end{enumerate}
\end{lemma}

\begin{example}\label{exam-a1} 
Let $X= \CC^n_z$ ($n \geq 2$) and for the two polynomials $P(z)=z_1, Q(z)
=z_2 \in \CC [z_1, \ldots, z_n]$ on it and $m \geq 1$ set 
\begin{equation}
f(z) := \frac{P(z)}{Q(z)^m} = \frac{z_1}{z_2^m} 
\quad (z \in X \setminus Q^{-1}(0)). 
\end{equation}
Then for a point $z \in X \setminus Q^{-1}(0)$ and 
$c \in \CC$ the condition $f(z)=c$ is equivalent to 
the one $z_1-cz_2^m=0$. By Theorem \ref{the-1} (iii), 
the distinguished triangle \eqref{dist} and 
\eqref{dist-isom}, this implies that for any 
point $z \in P^{-1}(0)= \{ z_1=0 \}$ we have 
$\phi_f^{\merc}( \CC_X )_z \simeq 0$.  
\end{example}
For some examples of $f: X \setminus Q^{-1}(0) \longrightarrow 
\CC$ such that $\phi_f^{\merc}( \CC_X ) \not\simeq 0$ and 
a method to calculate their stalks, 
see \cite[Section 2]{Tak-4}.

\section{Main theorems}\label{section 3}
In this section, for two polynomials $P(z), Q(z) \in 
\CC [z_1, \ldots, z_n]$ ($n \geq 2$) coprime each other 
on $X= \CC_z^n$ and $m \geq 1$, we consider a rational function 
\begin{equation}
f := \frac{P}{Q^m}: X \setminus Q^{-1}(0)
 \longrightarrow \CC \quad (z \longmapsto \frac{P(z)}{Q(z)^m}).
\end{equation}
Then we define the tameness at infinity of $f$ as follows. 

\begin{definition}\label{dfn:tame-rat}
Let $\partial f\colon X \setminus Q^{-1}(0) \longrightarrow \CC^n$ 
be the map defined by $\partial f(z)=(\partial_1f(z), \ldots, \partial_n f(z))$ 
($z \in X \setminus Q^{-1}(0)$). 
Then we say that $f$ is tame at infinity if there exist 
$0< \e \ll 1$ and $r \gg 0$ such that 
$(\partial f)^{-1}(B(0;\e )) \subset B(0;r) \setminus Q^{-1}(0)$. 
\end{definition}
As in the case of polynomial functions on $X= \CC^n$ we can 
show that the generic linear perturbations of $f$ are tame at infinity. 
In what follows, assume that $Q: X = \CC^n \longrightarrow \CC$ and 
$f: X \setminus Q^{-1}(0) \longrightarrow \CC$ are tame at 
infinity. This implies that 
$f^{-1}(0)=P^{-1}(0) \setminus Q^{-1}(0) \subset 
X \setminus Q^{-1}(0)$ and $Q^{-1}(0) \subset X= \CC^n$ have 
only isolated singular points. Assume also 
that $P^{-1}(0) \subset X= \CC^n$ has only 
isolated singular points and intersects $Q^{-1}(0)$ 
transversally outside the set $\{ p_1,p_2, \ldots, p_k \} 
\subset P^{-1}(0) \cap Q^{-1}(0)$ of the (isolated) 
singular points $p_1,p_2, \ldots, p_k$ of $P^{-1}(0)$ or $Q^{-1}(0)$ 
lying on $P^{-1}(0) \cap Q^{-1}(0)$. We set $U := X \setminus Q^{-1}(0)$.

\begin{lemma}\label{vc-lem} 
In the situation as above, for any $c \in \CC$ we have a vanishing 
\begin{equation}
H^j \phi_{t-c} {\rm R} f_! \CC_U \simeq 0 \quad (j \not= n-1).  
\end{equation}
\end{lemma} 

\begin{proof}
By the tameness at infinity of $f$, 
there exist $r_0 \gg 0$ and a small disk 
$D_{\eta_0}(c) := \{ t \in \CC \ | \ 
 |t-c| < \eta_0 \} \subset \CC$ ($0< \eta_0 \ll 1$) 
in $\CC$ centered at the point $c \in \CC$ such that 
for any $r \geq r_0$ and $t \in D_{\eta_0}(c)$ 
the sphere $\partial B(0;r)$ intersects $f^{-1}(t)$ 
transversally (see e.g. \cite[Proposition 3.4]{Tak-4}). 
Moreover, by the curve selection lemma at infinity 
in \cite[Lemma 2]{N-Z-2} we may assume also that 
for any $r \geq r_0$ 
the sphere $\partial B(0;r)$ intersects 
$P^{-1}(0) \cap Q^{-1}(0)$ transversally. 
Then by Example \ref{exam-a1} and 
the non-characteristic deformation lemma \cite[Proposition 2.7.2]{K-S}, 
this implies that we have isomorphisms 
\begin{align} 
i_c^{-1} {\rm R}f_! \CC_{U \cap B(0;r_0)} & \simeq 
i_c^{-1} {\rm R}f_! \CC_X, \\ 
\psi_{t-c} {\rm R}f_! \CC_{U \cap B(0;r_0)} & \simeq 
\psi_{t-c} {\rm R}f_! \CC_X. 
\end{align}
We thus obtain an isomorphism 
\begin{equation} 
 \phi_{t-c} {\rm R} f_! \CC_{U \cap B(0;r_0)}  \simeq 
\phi_{t-c} {\rm R} f_! \CC_U. 
\end{equation}
Let $q_1, \ldots, q_m \in f^{-1}(c)$ be the 
(isolated) singular points of $f^{-1}(c) \subset 
U=X \setminus Q^{-1}(0)$ and $\rho : X \times \CC 
\longrightarrow \CC$ the projection. Then for 
the (not necessarily) closed embedding 
\begin{equation}
i_f: U=X \setminus Q^{-1}(0) \hookrightarrow 
X \times \CC_t \qquad (x \longmapsto (x,f(x))).  
\end{equation}
we have $f= \rho \circ i_f$ and hence 
\begin{equation}
{\rm R} f_! \CC_{U \cap B(0;r_0)}  
\simeq 
{\rm R} \rho_* i_{f!} \CC_{U \cap B(0;r_0)}, 
\end{equation}
where we used the fact that $\rho$ is proper 
on the support of $i_{f!} \CC_{U \cap B(0;r_0)}$. 
By Proposition \ref{prp.7.1.002} and the definition of 
the meromorphic 
vanishing cycle functor $\phi_{f-c}^{\merc} ( \cdot )$ 
with compact support, we thus obtain an isomorphism 
\begin{equation} 
\phi_{t-c} {\rm R}f_! \CC_{U \cap B(0;r_0)} \simeq 
{\rm R}\Gamma (X ; \phi_{f-c}^{\merc} \CC_{B(0;r_0)}).  
\end{equation}
Moreover, by Example \ref{exam-a1} for $r_0 \gg 0$ 
we see that the 
support of $\phi_{f-c}^{\merc} \CC_{B(0;r_0)}$ is 
contained in the finite set 
\begin{equation} 
\{ p_1, \ldots, p_k \} \sqcup \{ q_1, \ldots, q_m \}
\quad \subset B(0;r_0) \subset X= \CC^n. 
\end{equation}
Then by the perversity of $\phi_{f-c}^{\merc} \CC_{B(0;r_0)}[n-1] 
\simeq \phi_{f-c}^{\merc} \CC_{X}[n-1]$ 
(see Theorem \ref{the-1} (ii)) we obtain the desired concentration 
\begin{equation}
H^j \phi_{t-c} {\rm R} f_! \CC_U \simeq 0 \quad (j \not= n-1).  
\end{equation}
\qed
\end{proof} 

Let $\Sigma \subset \CC$ be the union of $f(  {\rm Sing } f)$ and  the singular support of 
${\rm R} f_! \CC_U$.  For $c \in \CC \setminus \Sigma$ we call 
$f^{-1}(c) \subset U= X \setminus Q^{-1}(0)$ a generic fiber of 
$f: U= X \setminus Q^{-1}(0) \longrightarrow \CC$. 
Then by applying the arguments in the proof of 
\cite[Lemma 8.5]{Sabbah-2} to $f: U=X \setminus Q^{-1}(0) \longrightarrow \CC$ 
we obtain the following result. 

\begin{theorem}\label{MainT-1}
In the situation as above, assume also that $n \geq 3$. Then 
for any generic fiber $f^{-1}(c) \subset U=X \setminus Q^{-1}(0)$ 
($c \in \CC \setminus \Sigma$) of $f: U=X \setminus Q^{-1}(0) \longrightarrow \CC$ 
we have 
\begin{equation}
H^j(f^{-1}(c); \CC ) \simeq 0 \quad (j \not= 0,1,n-1) 
\end{equation}
and $H^0(f^{-1}(c); \CC ) \simeq H^1(f^{-1}(c); \CC ) \simeq \CC$. 
\end{theorem}

\begin{proof}
By the (generalized) Poincar\'e duality for the complex 
manifold $f^{-1}(c)$ it suffices to show that we have 
\begin{equation}
H^j_c(f^{-1}(c); \CC ) \simeq 0 \quad (j \not= n-1, 2n-3, 2n-2) 
\end{equation}
and $H^{2n-3}_c(f^{-1}(c); \CC ) \simeq H^{2n-2}_c(f^{-1}(c); \CC ) \simeq \CC$. 
Since for any $a \in \CC$ its fiber $f^{-1}(a)= 
(P-aQ)^{-1}(0) \setminus Q^{-1}(0) \subset U=X \setminus Q^{-1}(0)$ of $f$ 
is an affine variety and $\CC_{f^{-1}(a)} [n-1] \in \Dbc (f^{-1}(a))$ 
is a perverse sheaf on it, by Artin's 
vanishing theorem for perverse sheaves we have 
\begin{equation}
H^j_c(f^{-1}(a); \CC ) \simeq 0 \quad (j \notin [n-1,2n-2]) 
\end{equation}
(see e.g. \cite[Corollary 5.2.19]{Dimca}). 
This implies that we have a 
vanishing 
\begin{equation}\label{van-di}
H^j {\rm R} f_! \CC_U \simeq 0 \quad (j \notin [n-1,2n-2]).  
\end{equation}
Note also that by the (generalized) Poincar\'e duality for the complex 
manifold $U= X \setminus Q^{-1}(0)$ Propositions 
\ref{prp.L-T} and \ref{prp.gen-L-T}  
imply that we have 
\begin{equation}
H^j_c(U; \CC ) \simeq 0 \quad 
(j \not= n,n+1,2n-1,2n) 
\end{equation}
and $H^{2n-1}_c(U; \CC ) \simeq H^{2n}_c(U; \CC ) \simeq \CC$. 
Let us calculate the cohomology groups 
$H_c^j(U; \CC )$ ($j \in \ZZ$) with compact supports of $U=X \setminus Q^{-1}(0)$ 
by the Leray type spectral sequence $E_r^{p,q}$ ($r \geq 1, p,q \in \ZZ$) 
such that 
\begin{equation}
E_2^{p,q}=H^p_c( \CC ; H^q {\rm R} f_! \CC_U) \quad (p,q \in \ZZ ). 
\end{equation}
By \cite[page 43]{GM} we see that $H^j {\rm R} f_! \CC_U$ is 
an algebraic constructible sheaf on $\CC$ for any $j \in \ZZ$. 
Moreover, as in the proof of \cite[Lemma 5.3]{Sabbah-2}, 
by Lemma \ref{vc-lem} we can also show that 
$H^j {\rm R} f_! \CC_U$ is a constant sheaf on $\CC$ for any $j \geq n+1$. 
For general algebraic constructible sheaves on $\CC$ we have the 
following result. 

\begin{lemma}\label{acs-lem} 
Let $\F$ be an algebraic constructible sheaf on the affine line $\CC$. 
Then we have 
\begin{equation}
H^j_c( \CC ; \F ) \simeq 0 \quad (j \notin [0,2]).  
\end{equation}
\end{lemma} 

\begin{proof}
Let $S \subset \CC$ be the singular support of $\F$ and 
consider the exact sequence
\begin{equation}
0 \longrightarrow \F_{\CC \setminus S} \longrightarrow 
\F \longrightarrow \F_S \longrightarrow 0. 
\end{equation}
Note that $S$ is a finite subset of $\CC$ and hence 
$\F_S$ is a skyscraper sheaf on $\CC$. Moreover, 
as $\CC \setminus S$ is an affine variety, by Artin's 
vanishing theorem for perverse sheaves (see e.g. \cite[Corollary 5.2.19]{Dimca}) 
we obtain 
\begin{equation}
H^j_c( \CC ; \F_{\CC \setminus S} ) 
\simeq H^j_c( \CC \setminus S ; \F|_{\CC \setminus S}) \simeq 
0 \quad (j \notin [1,2]).  
\end{equation}
Then the assertion immediatedly follows. 
\qed
\end{proof} 
Summarizing these results, by Lemma \ref{acs-lem} we see that 
$E_2^{p,q}=0$ unless $(p,q) \in \ZZ^2$ is contained in the set 
\begin{equation}
\{ (0,n-1), (1,n-1), (2,n-1), (0,n), (1,n), (2,n), 
(2,n+1), \ldots (2, 2n-2) \} \quad \subset \ZZ^2. 
\end{equation}
This in particular implies that for $(p,q)=(2,n)$ the 
spectral sequence $E_r^{2,n}$ ($r \geq 1$) degenerates at $E_2$ 
i.e. $E_2^{2,n}=E_{\infty}^{2,n}$ and hence there exists 
an isomorphism 
\begin{equation}
E_2^{2,n}=H^2_c( \CC ; H^n {\rm R} f_! \CC_U) \simeq H^{n+2}_c(U; \CC ). 
\end{equation}
Note that if $n>3 \Longleftrightarrow 2n-1>n+2$ (resp. $n=3$) 
then we have $H^{n+2}_c(U; \CC ) \simeq 0$ (resp. 
$H^{n+2}_c(U; \CC )=H^{2n-1}_c(U; \CC ) \simeq \CC$). 
Let $S \subset \Sigma \subset \CC$ be the singular support of the 
algebraic constructible sheaf $\F_n :=
H^n {\rm R} f_! \CC_U$ on $\CC$ and denote the 
local system $\F_n |_{\CC \setminus S}$ on $\CC \setminus S$ by 
$\L$. For $c \in \CC$ let $i_c: \{ c \} \hookrightarrow \CC$ 
be the inclusion map. Then as in the proof of \cite[Lemma 8.5]{Sabbah-2} 
we obtain an exact sequence 
\begin{equation}
i_c^{-1} \F_n \longrightarrow \psi_{t-c} \F_n \longrightarrow 0. 
\end{equation}
This implies that the monodromy of the local system 
$\L$ around each point of $S$ is trivial and hence 
$\L$ is a constant sheaf on $\CC \setminus S$. By the proof of 
Lemma \ref{acs-lem} we thus find that the dimension of 
\begin{equation}
E_2^{2,n}=H^2_c( \CC ; \F_n ) \simeq H^{2}_c( \CC \setminus S; \L )
\end{equation}
is equal to the rank of $\L$ i.e. the dimension of 
$H_c^n(f^{-1}(c); \CC )$ for the generic fiber $f^{-1}(c) 
\subset X \setminus Q^{-1}(0)$ of $f$. In particular, 
if $n>3$ (resp. $n=2$) we have $H_c^n(f^{-1}(c); \CC ) \simeq 0$ 
(resp. $H_c^n(f^{-1}(c); \CC )= H_c^{2n-3}(f^{-1}(c); \CC ) 
\simeq \CC$). For $j \geq n+1$ and $(p,q)=(2,j)$ the spectral sequence 
$E_r^{2,j}$ ($r \geq 1$) degenerates at $E_2$ 
i.e. $E_2^{2,j}=E_{\infty}^{2,j}$ and hence there exists 
an isomorphism 
\begin{equation}
E_2^{2,j}=H^2_c( \CC ; H^j {\rm R} f_! \CC_U) \simeq H^{j+2}_c(U; \CC ). 
\end{equation}
Since for $j \geq n+1$ the algebraic constructible sheaf 
$H^j {\rm R} f_! \CC_U$ on $\CC$ is constant, we see also that 
the dimension of 
$E_2^{2,j}=H^2_c( \CC ; H^j {\rm R} f_! \CC_U) \simeq H^{j+2}_c(U; \CC )$ 
is equal to that of $H_c^j(f^{-1}(c); \CC )$ 
for the generic fiber $f^{-1}(c) 
\subset X \setminus Q^{-1}(0)$ of $f$. For $j \geq n+1$ and the 
generic fiber $f^{-1}(c)$ of $f$ then there exist isomorphisms 
\begin{equation}
H^j_c( f^{-1}(c); \CC ) \simeq 
\begin{cases}
\ \CC & (j=2n-3, 2n-2) \\
& \\
\ 0 & (j \not= 2n-3,2n-2) 
\end{cases}
\end{equation}
(if $n=3$ the case where $j \geq n+1=4$ and $j=2n-3=3$ does 
not occur). For the generic fiber $f^{-1}(c) 
\subset X \setminus Q^{-1}(0)$ of $f$ we thus obtain 
\begin{equation}
H^j_c(f^{-1}(c); \CC ) \simeq 0 \quad (j \not= n-1, 2n-3, 2n-2) 
\end{equation}
and $H^{2n-3}_c(f^{-1}(c); \CC ) \simeq H^{2n-2}_c(f^{-1}(c); \CC ) \simeq \CC$. 
This completes the proof. 
\qed
\end{proof}

\begin{remark}
In the case $n=2$, if a generic fiber $f^{-1}(c) 
\subset X \setminus Q^{-1}(0)$ ($c \in \CC$) of $f$ is connected then 
it is a connected Riemann surface 
with boundary and hence we have 
\begin{equation}
H^j(f^{-1}(c); \CC ) \simeq 0 \quad (j \not= 0,1) 
\end{equation}
and $H^{0}(f^{-1}(c); \CC ) \simeq \CC$. 
\end{remark} 
Motivated by Theorems \ref{tame} and \ref{MainT-1}, we 
propose the following conjectures. 

\begin{conjecture}\label{conj-1}
In the situation of Theorem \ref{MainT-1}, for 
any fiber $f^{-1}(c) \subset U=X \setminus Q^{-1}(0)$ 
($c \in \CC$) of $f: U=X \setminus Q^{-1}(0) \longrightarrow \CC$ 
we have 
\begin{equation}
H^j(f^{-1}(c); \CC ) \simeq 0 \quad (j \not= 0,1,n-1) 
\end{equation}
and $H^0(f^{-1}(c); \CC ) \simeq H^1(f^{-1}(c); \CC ) \simeq \CC$. 
\end{conjecture}

\begin{conjecture}\label{conj-2}
In the situation of Theorem \ref{MainT-1}, 
any generic fiber $f^{-1}(c) \subset U=X \setminus Q^{-1}(0)$ 
($c \in \CC \setminus \Sigma$) of $f: U=X \setminus Q^{-1}(0) \longrightarrow \CC$ 
has the homotopy type of 
\begin{equation}
S^1 \vee (S^{n-1} \vee \cdots \vee S^{n-1}).
\end{equation}
\end{conjecture}

We denote by $\mu \geq 0$ the total Milnor number of the tame function 
$Q:X= \CC^n \longrightarrow \CC$. For a point 
$a \in \CC$ we denote by $\mu_a \geq 0$ the sum of the Milnor numbers of 
$Q:X= \CC^n \longrightarrow \CC$ at the (isolated) singular points of 
$Q^{-1}(a) \subset X= \CC^n$. 
Then in the situation of Theorem \ref{MainT-1} 
we can also give a formula for the dimension of 
$H^{n-1}(f^{-1}(c); \CC )$ ($c \in \CC \setminus \Sigma$) 
as follows. In the course of the proof of 
Theorem \ref{MainT-2} below, we will give also a new proof to 
the results in Theorem \ref{MainT-1}. Let $c_1, \ldots, c_l \in \CC$ 
be the points of $\Sigma \subset \CC$ and for $1 \leq i \leq l$ 
let $q_{i1}, \ldots, q_{im_i} \in f^{-1}(c_i)$ be the 
(isolated) singular points of $f^{-1}(c_i) \subset U= 
X \setminus Q^{-1}(0)$. Then by the proof of Lemma \ref{vc-lem} 
for any $1 \leq i \leq l$ there exists an isomorphism 
\begin{equation}\label{eff-tool}
H^{n-1} \phi_{t-c_i} {\rm R} f_! \CC_U \simeq 
\Bigl( \bigoplus_{j=1}^{m_i} H^{n-1} (\phi_{f-c_i} \CC_U)_{q_{ij}} \Bigr) 
\oplus 
\Bigl( \bigoplus_{j=1}^{k} H^{n-1} (\phi_{f-c_i}^{\merc} \CC_X)_{p_{j}} \Bigr). 
\end{equation}

\begin{theorem}\label{MainT-2}
In the situation of Theorem \ref{MainT-1}, 
for any generic fiber $f^{-1}(c) \subset U=X \setminus Q^{-1}(0)$ 
($c \in \CC \setminus \Sigma$) of $f: U=X \setminus Q^{-1}(0) \longrightarrow \CC$ 
we have 
\begin{equation}
\dim H^{n-1}(f^{-1}(c); \CC ) = \mu_f-( \mu - \mu_0),  
\end{equation}
where we set 
\begin{equation}
\mu_f := \sum_{i=1}^l \dim  H^{n-1} \phi_{t-c_i} {\rm R} f_! \CC_U. 
\end{equation}
In particular, there exists an inequality 
\begin{equation}
\mu_f \geq \mu - \mu_0.  
\end{equation}
\end{theorem}

\begin{proof}
We define a function $\varphi : \CC \longrightarrow \RR$ by 
$\varphi (t):= |t-c|$ ($t \in \CC$). Then we prove the assertion 
by applying a backward Morse theory to ${\rm R} f_! \CC_U$ and $\varphi$. 
First of all, for sufficiently large $r \gg 0$ there exists an isomorphism 
\begin{equation}
{\rm R}\Gamma_c( \{ \varphi <r  \} ; {\rm R} f_! \CC_U) 
\simeq {\rm R}\Gamma_c(U; \CC_U)
\end{equation}
and hence by Propositions \ref{prp.L-T} and \ref{prp.gen-L-T} we have 
\begin{equation}
{\rm R}\Gamma_c( \{ \varphi <r  \} ; {\rm R} f_! \CC_U)  \simeq 0 \quad 
(j \not= n,n+1,2n-1,2n-2) 
\end{equation}
and $H^{2n-1}_c( \{ \varphi <r  \} ; {\rm R} f_! \CC_U) \simeq 
H^{2n}_c( \{ \varphi <r  \} ; {\rm R} f_! \CC_U) \simeq \CC$. Moreover we have 
\begin{equation}
\dim H^n_c ( \{ \varphi <r  \} ; {\rm R} f_! \CC_U) - 
\dim H^{n+1}_c ( \{ \varphi <r  \} ; {\rm R} f_! \CC_U)= 
\mu - \mu_0. 
\end{equation}
For a fixed $r>0$ and $0< \e \ll 1$ let us consider how the 
cohomology groups of 
${\rm R}\Gamma_c( \{ \varphi <r- \e \} ; {\rm R} f_! \CC_U)$ 
change from those of 
${\rm R}\Gamma_c( \{ \varphi <r+ \e \} ; {\rm R} f_! \CC_U)$. 
For this purpose, we can use the distinguished triangle 
\begin{align}\label{d-t-3} 
\begin{split} 
& {\rm R}\Gamma_c( \{ \varphi <r- \e \} ; {\rm R} f_! \CC_U)
 \longrightarrow {\rm R}\Gamma_c ( \{ \varphi <r+ \e \} ; {\rm R} f_! \CC_U)
\\ 
& \longrightarrow 
{\rm R}\Gamma_c( \{ r- \e \leq \varphi <r+ \e   \} ; {\rm R} f_! \CC_U)
\overset{+1}{\longrightarrow}. 
\end{split} 
\end{align}
For $1 \leq i \leq l$ such that $\varphi (c_i)=|c_i-c|=r$ and 
$\delta >0$ such that $0< \e < \delta \ll 1$ we set 
$D_{\delta}(c_i):= \{ t \in \CC \ | \ |t-c_i| < \delta \} 
\subset \CC$. Then we define a closed subset 
$K \subset \CC$ by $K:= \sqcup_{i: \varphi (c_i)=r} 
\overline{D_{\delta}(c_i)}$ to obtain a 
distinguished triangle 
\begin{align}
\begin{split}
& {\rm R}\Gamma_c( \{ r- \e \leq \varphi <r+ \e   \} \setminus K ; {\rm R} f_! \CC_U)
 \longrightarrow {\rm R}\Gamma_c ( \{  r- \e \leq \varphi <r+ \e \} ; {\rm R} f_! \CC_U)
\\ 
& \longrightarrow 
{\rm R}\Gamma_c( \{ r- \e \leq \varphi <r+ \e   \} \cap K ; {\rm R} f_! \CC_U)
\overset{+1}{\longrightarrow}. 
\end{split}
\end{align}
We can easily see that 
\begin{equation}
{\rm R}\Gamma_c( \{ r- \e \leq \varphi <r+ \e   \} \setminus K ; {\rm R} f_! \CC_U)
\simeq 
{\rm R}\Gamma_c( \CC ; (
{\rm R} f_! \CC_U)_{\{ r- \e \leq \varphi <r+ \e   \} \setminus K} )
\simeq 0. 
\end{equation}
Moreover, for any $1 \leq i \leq l$ such that $\varphi (c_i)=|c_i-c|=r$ 
there exist isomorphisms 
\begin{align}
\begin{split}
& {\rm R}\Gamma_c( \{ r- \e \leq \varphi <r+ \e   \} \cap \overline{D_{\delta}(c_i)}
 ; {\rm R} f_! \CC_U)
\\
& \simeq 
{\rm R}\Gamma_c( \CC ; (
{\rm R} f_! \CC_U)_{\{ r- \e \leq \varphi <r+ \e   \} 
\cap \overline{D_{\delta}(c_i)}} )
\simeq 
{\rm R}\Gamma_{\{ \varphi \leq r \}}( \overline{D_{\delta}(c_i)}; 
{\rm R} f_! \CC_U)
\\
& \simeq 
\phi_{t-c_i} {\rm R} f_! \CC_U [-1].
\end{split}
\end{align}
By Lemma \ref{vc-lem}, this implies that we have 
\begin{equation}
H^j_c( \{ r- \e \leq \varphi <r+ \e   \} ; {\rm R} f_! \CC_U)
\simeq 0 \quad (j \not= n) 
\end{equation}
and 
\begin{equation}
\dim H^n_c( \{ r- \e \leq \varphi <r+ \e   \} ; {\rm R} f_! \CC_U)
= \sum_{i: \varphi (c_i)=r} \dim H^{n-1} \phi_{t-c_i} {\rm R} f_! \CC_U. 
\end{equation}
On the other hand, for $0<r \ll 1$ it follows from the 
condition $c \notin \Sigma$ that there exists an isomorphism 
\begin{equation}
{\rm R}\Gamma_c( \{ \varphi <r  \} ; {\rm R} f_! \CC_U)[2] 
\simeq ({\rm R} f_! \CC_U)_c. 
\end{equation}
By \eqref{van-di} for $0<r \ll 1$ we thus obtain a vanishing 
\begin{equation}
H^j_c( \{ \varphi <r  \} ; {\rm R} f_! \CC_U) 
\simeq 0 \quad (j \notin [n+1, 2n]). 
\end{equation}
Combining these results together, 
the assertion immediately follows from the 
cohomology long exact sequence associated to 
the distinguished triangle \eqref{d-t-3}. 
This completes the proof. 
\qed
\end{proof} 
Finally, motivated by Theorem \ref{MainT-2}, we 
strengthen Conjecture \ref{conj-2} as follows.

\begin{conjecture}\label{conj-3}
In the situation of Theorem \ref{MainT-2}, 
any generic fiber $f^{-1}(c) \subset U=X \setminus Q^{-1}(0)$ 
($c \in \CC \setminus \Sigma$) of $f: U=X \setminus Q^{-1}(0) \longrightarrow \CC$ 
has the homotopy type of 
\begin{equation}
S^1 \vee (S^{n-1} \vee \cdots \vee S^{n-1}), 
\end{equation}
where $S^{n-1} \vee \cdots \vee S^{n-1}$ is 
the bouquet of the $\mu_f-( \mu - \mu_0)$ copies of the 
$(n-1)$-sphere $S^{n-1}$.
\end{conjecture}



\begin{thebibliography}{99}

\bibitem{A-G-L-N} 
{\`A}lvarez Montaner, J., Gonz{\'a}lez Villa, M., 
Le{\'o}n-Cardenal, E. and N{\'u\~n}ez-Betancourt, L., 
\emph{Bernstein-Sato polynomial and related invariants 
for meromorphic functions}, 
Trans. Amer. Math. Soc., 378, No. 7 (2025): 4929-4954. 

\bibitem{ALM-1}
Artal Bartolo, E., Luengo, I. and Melle-Hern\'andez, A. 
\emph{Milnor number at infinity, topology and 
Newton boundary of a polynomial function}, 
Mathematische Zeitschrift, 233 (2000): 679-696. 

\bibitem{B-P}
Bodin, A. and Pichon, A. 
\emph{Meromorphic functions, bifurcation sets 
and fibred links}, Math. Res. Lett., 
14, No. 3 (2007): 413-422. 

\bibitem{B-P-S}
Bodin, A., Pichon, A. and Seade, J. 
\emph{Milnor fibrations of meromorphic functions}, 
J. London Math. Soc., 
80, No. 2 (2009): 311-325. 

\bibitem{Broughton}
Broughton, S. A. \emph{Milnor numbers and the topology of 
polynomial hypersurfaces}, Invent. Math., 92 (1988): 217-241.

\bibitem{C-R}
Cassou-Nogu\`es, P. and Raibaut,  M.  
\emph{Newton transformations and motivic invariants at infinity of plane curves}, 
Mathematische Zeitschrift, 299 (2021): 591-669.

\bibitem{Deligne}
Deligne, P. 
\emph{Le formalisme des cycles 
\'evanescents, in SGA7 XIII and XIV}, 
Lecture Notes in Math., Springer-Verlag, Berlin,, vol. 340,
(1973), 82--115 and 116--164.

\bibitem{Dimca2}
Dimca, A. \emph{Monodromy at infinity for polynomials in two variables}, 
J. Alg. Geom., 7 (1998): 771-779.

\bibitem{Dimca}
Dimca, A. \emph{Sheaves in topology}, Universitext, Springer-Verlag, Berlin, 2004.

\bibitem{D-N}
Dimca, A. and N{\'e}methi, A. 
\emph{On the monodromy of complex polynomials}, 
Duke Math. J., 108 (2001): 199-209.

\bibitem{D-S-1}
Dimca, A. and Saito, M. 
\emph{Monodromy at infinity and the weights of cohomology}, 
Compositio Math., 138 (2003): 55-71.

\bibitem{Douai}
Douai, A. \emph{A note on the (toric) Newton spectrum of a polynomial}, 
arXiv:1810.03901, preprint. 

\bibitem{E-T}
Esterov, A. and Takeuchi, K. 
\emph{Motivic Milnor fibers over complete 
intersection varieties and their virtual 
Betti numbers}, Int. Math. Res. Not., 
Vol. 2012, No. 15 (2012): 3567-3613. 

\bibitem{L-N-2}
Garc{\'i}a L{\'o}pez, R. and N{\'e}methi, A. 
\emph{Hodge numbers attached to a polynomial map}, 
Ann. Inst. Fourier, 49 (1999): 1547-1579.

\bibitem{GM}
Goresky, M. and MacPherson, R.  
\emph{Stratified Morse Theory},
Ergebnisse der Mathematik und ihrer Grenzgebiete (3), vol. 14,
Springer-Verlag, Berlin, 1988.

\bibitem{G-L-M}
Guibert, G., Loeser, F. and Merle, M. 
\emph{Iterated vanishing cycles, convolution, and a motivic analogue of a 
conjecture of Steenbrink}, Duke Math. J., 132 (2006): 409-457.

\bibitem{GZ-L-MH}
Gusein-Zade, S., Luengo, I. and Melle-Hern\'andez, A. 
\emph{Zeta-functions for germs of meromorphic 
functions and Newton diagrams}, 
Funct. Anal. Appl., 32, No. 2 (1998): 26-35.

\bibitem{GZ-L-MH-new}
Gusein-Zade, S., Luengo, I. and Melle-Hern\'andez, A. 
\emph{On the zeta-function of a polynomial 
at infinity}, Bull. Sci. Math., 124, No. 3 (2000): 213-224.  

\bibitem{GZ-L-MH-new-new}
Gusein-Zade, S., Luengo, I. and Melle-Hern\'andez, A. 
\emph{Bifurcations and topology of meromorphic germs}, 
New developments in singularity theory 
(Cambridge 2000), volume 21 of NATO Sci. Ser. 
II Math. Phys. Chem., pages 279-304.
Kluwer Acad. Publ., Dordrecht, 2001. 

\bibitem{H-T-T}
Hotta, R., Takeuchi, K. and Tanisaki, T. 
\emph{D-modules, perverse sheaves, and 
representation theory}, Birkh{\"a}user Boston, 2008.

\bibitem{K-S}
Kashiwara, M. and Schapira, P. \emph{Sheaves on manifolds}, Springer-Verlag, 1990.

\bibitem{Ka-St-2}
Katz, E. and Stapledon, A. 
\emph{Tropical geometry, the motivic nearby fiber 
and limit mixed Hodge numbers of hypersurfaces},
Res. Math. Sci., 3:10 (2016): 36pp.

\bibitem{L-S}
Libgober, A. and Sperber, S. \emph{On the zeta function of monodromy of
 a polynomial map}, Compositio Math., 95 (1995): 287-307.

\bibitem{L-T}
Libgober, A. and Tib{\u a}r, M. \emph{Homotopy groups of complements 
and nonisolated singularities}, Int. Math. Res. Not., 
Vol. 2002, No. 17 (2002): 871-888. 

\bibitem{M-T-new3}
Matsui, Y. and Takeuchi,  K.  \emph{Monodromy zeta functions at infinity, 
Newton polyhedra and constructible sheaves}, 
Mathematische Zeitschrift, 268 (2011): 409-439.

\bibitem{M-T-3}
Matsui, Y. and Takeuchi, K. \emph{Monodromy 
at infinity of polynomial maps and 
Newton polyhedra, with Appendix by C. Sabbah}, 
Int. Math. Res. Not., 
Vol. 2013, No. 8 (2013): 1691-1746. 

\bibitem{M-T-2}
Matsui, Y. and Takeuchi, K. \emph{On the sizes of the Jordan blocks 
of monodromies at infinity}, 
Hokkaido Math. Journal, 44, (2015): 313-326. 

\bibitem{M-M}
Maxim, L. and Messina, K. \emph{On the topology of fibers of complex polynomial maps},
arXiv:2608.04320, (2026). 

\bibitem{Milnor}	
Milnor, J. \emph{Singular points of complex hypersurfaces}, 
Princeton University Press, 1968.


\bibitem{N-Z-2}
N{\'e}methi, A. and Zaharia, A. 
\emph{Milnor fibration at infinity}, 
Indag. Math., 
3, no.3 (1992): 323-335.

\bibitem{N-N}
Neumann, W. D. and Norbury, P. \emph{Vanishing cycles and monodromy of 
complex polynomials}, Duke Math. J., 101 (2000): 487-497.

\bibitem{Thang} 
Nguyen, T. T., \emph{On the topology of rational
functions in two complex variables}, 
Acta. Math. Vietnam. 37(2) (2012): 171-187.

\bibitem{N-S-T} 
Nguyen, T. T., Saito T. and Takeuchi, K., 
\emph{The bifurcation set of a rational function 
via Newton polytopes}, 
Mathematische Zeitschrift, 
298 (2021): 899-916.

\bibitem{N-T} 
Nguyen, T. T. and Takeuchi, K., 
\emph{Meromorphic nearby cycle functors and 
monodromies of meromorphic functions 
(with Appendix by T. Saito)}, 
Revista Matematica Complutense, 36 
(2023): 663-705. 

\bibitem{Paru} 
Parusinski, A., \emph{A note on singularities at infinity 
of complex polynomials}, in 
Symplectic singularities and geometry of gauge fields, 
Banach Center Publ., 39 (1997): 131-141. 

\bibitem{Raibaut-1.5}
Raibaut, M. \emph{Singularit\'es  \`a l'infini 
et int\'egration motivique}, Bull. Soc. Math. France, 
140 (1), (2012): 51-100. 

\bibitem{Raibaut-1.8}
Raibaut, M. \emph{Fibre de Milnor motivique \`a 
 l'infini et composition avec un polyn\^ome non 
d\'eg\'en\'er\'e}, Ann. Inst. Fourier, 
62, no.5 (2012): 1943-1981.  

\bibitem{Raibaut-2}
Raibaut, M. \emph{Motivic Milnor fibers of a rational function}, 
Revista Matematica Complutense, 
26, (2013): 705-734. 

\bibitem{Sabbah}
Sabbah, C. \emph{Monodromy at infinity and Fourier transform}, 
Publ. Res. Inst. Math. Sci., 33 (1997): 643-685.

\bibitem{Sabbah-2}
Sabbah, C. \emph{Hypergeometric periods for a tame polynomial}, 
Port. Math., 63 (2006): 173-226.

\bibitem{S-T}
Saito, T. and Takeuchi, K. 
\emph{On the monodromies and the limit mixed 
Hodge structures of families of algebraic varieties}, 
Michigan Mathematical Journal, 73, no.3 
(2023): 631-668. 

\bibitem{Sch}
Sch\"urmann, J. \emph{Topology of singular spaces and 
constructible sheaves}, Mathematics 
Institute of the Polish Academy of Sciences. 
Mathematical Monographs (New Series), 63. 
Birkh{\"a}user Basel, 2003.

\bibitem{S-T-1}
Siersma, D. and Tib{\u a}r, M. \emph{Singularities at infinity and their 
vanishing cycles}, Duke Math. J., 80 (1995): 771-783.

\bibitem{S-T-2}
Siersma, D. and Tib{\u a}r, M. \emph{
Vanishing cycles and singularities of meromorphic 
functions}, Contemp. Math., 354 (2004): 277-289.

\bibitem{Stapledon} 
Stapledon, A. \emph{Formulas for monodromy}, 
Res. Math. Sci., 4:8 (2017): 42pp.

\bibitem{Tak-2}
Takeuchi, K. \emph{On a Bernstein-Sato polynomial of a meromorphic function}, 
Nagoya Math. Journal, 251 (2023): 715-733.

\bibitem{Tak-3}
Takeuchi, K. \emph{Geometric monodromies, mixed Hodge numbers of motivic Milnor fibers and Newton polyhedra},
Handbook of Geometry and Topology of Singularities VII, (2025),  643-720, Springer

\bibitem{Tak-4}
Takeuchi, K. \emph{Fourier transforms of irregular holonomic D-modules, 
singularities at infinity of meromorphic functions 
and irregular characteristic cycles}, 
arXiv:2211.04113, preprint. 

\bibitem{T-T}
Takeuchi, K. and Tib{\u a}r, M. \emph{Monodromies 
at infinity of non-tame polynomials}, 
Bull. Soc. Math. France, 
144, no.3 (2016): 477-506.

\bibitem{Tibar}
Tib{\u a}r, M. \emph{Singularities and topology of meromorphic functions}, 
in Trends in singularities, Birkh{\"a}user Verlag (2002): 223-246.

\bibitem{Tibar-book}
Tib{\u a}r, M. 
\emph{Polynomials and vanishing cycles}, 
Cambridge University Press, 2007.



\end{thebibliography}
\end{document}